\documentclass[12pt,reqno]{amsart}

\usepackage{amsmath,amssymb,amsfonts}
\usepackage[abbrev]{amsrefs}
\usepackage{hyperref,enumerate,verbatim,mathrsfs}
\usepackage{tikz-cd}

\usepackage{microtype}
\usepackage{color,hyperref}

\definecolor{darkblue}{rgb}{0.0,0.0,0.3}
\hypersetup{colorlinks,breaklinks,
            linkcolor=red,urlcolor=red,
            anchorcolor=red,citecolor=red}

\theoremstyle{plain}
\newtheorem{theorem}{Theorem}  
\newtheorem*{theorem*}{Theorem}
\newtheorem*{corollary*}{Corollary}
\newtheorem{cor}[theorem]{Corollary}
\newtheorem{lemma}[theorem]{Lemma}

\newtheorem{prop}[theorem]{Proposition}
\numberwithin{equation}{section}

\theoremstyle{definition}
\newtheorem{defi}[theorem]{Definition}
\newtheorem{rem}[theorem]{Remark}
\newtheorem{exam}[theorem]{Example}
\numberwithin{theorem}{section}
\newcommand{\m}[1]{\mathbb{#1}}

\newcommand{\ms}[1]{\mathscr{#1}}
\newcommand{\mr}[1]{\mathrm{#1}}
\newcommand{\mc}[1]{\mathcal{#1}}

\newcommand{\CC}{{\m{C}}}

\newcommand{\eps}{\varepsilon}
\newcommand{\s}{\subseteq}

\newcommand{\mA}{\mc{A}}

\newcommand{\mM}{\ms{M}}

\newcommand{\mZ}{\ms{Z}}

\newcommand{\idd}{\mr{id}}
\newcommand{\mult}{\mr{mult}}
\newcommand{\Aut}{\mr{Aut}}

\newcommand{\nl}{\vspace{12pt}\noindent}

\newcommand{\Span}{\mathrm{span}}

\newcommand{\ov}{\overline}
\newcommand{\sm}{\setminus}
\newcommand{\oo}{\emptyset}

\begin{document}

\title[Injective envelopes and the intersection property]{Injective envelopes\\and the intersection property}
\author[R. S. Bryder]{Rasmus Sylvester Bryder}
\address{Department of Mathematics\\ University of Copenhagen\\
U\-ni\-ver\-sitets\-parken 5, 2100 Copenhagen, Denmark}
\email{rbryder@gmail.com}

\begin{abstract}
We consider the ideal structure of a reduced crossed product of a unital $C^*$-algebra equipped with an action of a discrete group. More specifically we find sufficient and necessary conditions for the group action to have the intersection property, meaning that non-zero ideals in the reduced crossed product restrict to non-zero ideals in the underlying $C^*$-algebra. We show that the intersection property of a group action on a $C^*$-algebra is equivalent to the intersection property of the action on the equivariant injective envelope. We also show that the centre of the equivariant injective envelope always contains a $C^*$-algebraic copy of the equivariant injective envelope of the centre of the injective envelope. Finally, we give applications of these results in the case when the group is $C^*$-simple.
\end{abstract}

\thanks{The author is supported by a PhD stipend from the Danish National Research Foundation (DNRF) through the Centre for Symmetry and Deformation at University of Copenhagen}
\maketitle

\section{Introduction}
In his paper \cite{hamana1979} from 1979, Hamana proved the existence and uniqueness of the \emph{injective envelope} $I(A)$ of a unital $C^*$-algebra $A$. The injective envelope is a unital $C^*$-algebra, which is injective in the category of unital $C^*$-algebras with unital, completely positive maps (the monomorphisms in said category being complete isometries), which is minimal with respect to containment of a $^*$-algebraic completely isometric copy of a given $C^*$-algebra.

Hamana later generalized his result \cite{hamana1985} to $G$-operator systems, i.e., operator systems with an action of a discrete group $G$ by complete order isomorphisms, proving that any such system $S$ has a unique $G$-injective envelope $I_G(S)$ in the category of $G$-operator systems and unital $G$-equivariant completely positive maps. In this paper we consider the structure of the $G$-injective envelope of a unital $C^*$-algebra with an action of a discrete group $G$.

Recently the notion of a $G$-injective envelope of a $G$-$C^*$-algebra has surfaced as a helpful device in uncovering the ideal structure of the reduced group $C^*$-algebra and reduced crossed products associated to a discrete group $G$. Perhaps the most celebrated application has been Kalantar and Kennedy's characterizations of simplicity of the reduced group $C^*$-algebra of a discrete group \cite{kalantarkennedy}. We say that a discrete group $G$ is \emph{$C^*$-simple} if its associated reduced group $C^*$-algebra is simple. Kalantar and Kennedy then proved that $G$ is $C^*$-simple if and only if the action of $G$ on the maximal ideal space of the $G$-injective envelope $I_G(\CC)$ is free.

The result of Kalantar and Kennedy was later generalized by Kawabe \cite{kawabe} who proved for any unital commutative $G$-$C^*$-algebra $A$ that the action of $G$ on $A$ satisfies the \emph{intersection property} if and only if the action of $G$ on the maximal ideal space of the commutative $C^*$-algebra $I_G(A)$ is free. The action of a discrete group $G$ on a $C^*$-algebra $A$ is said to satisfy the intersection property if every non-zero ideal in the associated reduced crossed product $A\rtimes_r G$ has non-zero intersection with $A$. This concept provides an important means of understanding ideals in reduced crossed products. Indeed, the reduced crossed product $A\rtimes_r G$ is simple if and only if $A$ has no non-trivial $G$-invariant ideals (i.e., the action is minimal) and the action of $G$ on $A$ has the intersection property. In \cite{kawamuratomiyama}, Kawamura and Tomiyama gave a characterization of the intersection property for amenable groups acting on unital commutative $C^*$-algebras. Archbold and Spielberg were among the first \cite{archboldsp} to give sufficient criteria for the property to hold even in the general non-commutative case, namely in terms of \emph{topological freeness} of the group action. Generally, however, the intersection property need not imply topological freeness.

The central question in this paper is whether Kawabe's result can be generalized to the non-commutative case. In general the structure of the injective envelope of a $C^*$-algebra is hard to discern, and even more so for the equivariant case. In Theorem~\ref{essip} we establish that the action of $G$ on a unital (not necessarily commutative) $C^*$-algebra $A$ has the intersection property if and only if the action of $G$ on any $G$-essential $C^*$-algebra extension $B$ of $A$ has the intersection property (see Definition \ref{essrig}). Note that the $G$-injective envelope of a $G$-$C^*$-algebra is a $G$-essential $C^*$-algebra extension of $A$ (see Section \ref{ext} below).

In order to find sufficient criteria for the action of a group on a unital $C^*$-algebra to have the intersection property, we consider the question of whether the centre of the $G$-injective envelope contains $G$-invariant $C^*$-sub\-al\-ge\-bras that may be more easily determined by the structure of the original $C^*$-algebra than the centre itself. It is \emph{not} possible in general to identify the centre of the $G$-injective envelope by means of the centre of the original $C^*$-algebra (see Remark \ref{countcount}). However, in Theorem \ref{main2} we prove for all unital $C^*$-algebras $A$ that the centre of the $G$-injective envelope $I_G(A)$ of $A$ contains the $G$-injective envelope of the centre of the injective envelope $I(A)$ of $A$ as a $G$-invariant unital $C^*$-subalgebra, or, more formally,
$$I_G(Z(I(A)))\s Z(I_G(A)).$$ We are unsure whether equality holds. 
 By combining this result with a result of Kawabe, we obtain the result (Theorem \ref{supercentre}) that the action of a discrete group $G$ on a unital $C^*$-algebra $A$ has the intersection property whenever the action of $G$ on the centre of the injective envelope $I(A)$ has the intersection property.

\nl

Any $C^*$-algebra is assumed to be unital (unless otherwise stated) and the centre of a $C^*$-algebra $A$ is denoted by $Z(A)$. Ideals in $C^*$-algebras are always assumed to be closed and two-sided. Moreover $G$ \emph{always} denotes a discrete group with identity element $1$.
\section{Preliminaries}
\subsection{Extensions of operator systems}\label{ext}Recall that an \emph{operator system} is a self-adjoint linear subspace $S$ of a unital $C^*$-algebra $A$ containing the identity $1$ of $A$. A \emph{complete order isomorphism} of two operator systems $S$ and $S'$ is a unital, completely positive (\mbox{u.c.p.}) linear isomorphism $\varphi\colon S\to S'$ such that $\varphi^{-1}$ is also completely positive. If $S'=S$, we say that $\varphi$ is an \emph{automorphism}, and we let $\Aut(S)$ denote the group of automorphisms of $S$. Any $\varphi\in\Aut(S)$ is automatically completely isometric, and further, if $S$ is a (unital) $C^*$-algebra, then the definition of an automorphism coincides with the usual notion of an automorphism of a $C^*$-algebra (e.g., see the beginning of the proof of \cite[Theorem~3.1]{choieffros}). 

An action of a discrete group $G$ on an operator system $S$ is always assumed to be by automorphisms, and in this setting we then say that $S$ is a \emph{$G$-operator system}. (If $S$ is a $C^*$-algebra, we will of course say that $S$ is a \emph{$G$-$C^*$-algebra}.) For any $g\in G$ and $x\in S$, the image of $x$ under $g$ is denoted by $gx$.

\begin{defi}A $G$-operator system $S$ is said to be \emph{$G$-injective} if for all $G$-operator systems $E$ and $F$, any $G$-equivariant complete isometry $\kappa\colon E\to F$  and any $G$-equivariant \mbox{u.c.p.} map $\varphi\colon E\to S$ there exists a $G$-equivariant \mbox{u.c.p.} map $\tilde{\varphi}\colon F\to S$ satisfying $\tilde{\varphi}\circ\kappa=\varphi$, i.e., the following diagram commutes:\[\begin{tikzcd}
E\arrow[r,hook,"\kappa"]\arrow[dr,"\varphi"]&F\arrow[d,dashed,"\tilde\varphi"]\\
&S\end{tikzcd}\]\end{defi}

A $G$-injective $G$-operator system is therefore an injective object in the category of $G$-operator systems and $G$-equivariant \mbox{u.c.p.} maps.

\begin{defi}\label{essrig}Let $S$ and $M$ be $G$-operator systems and let $\kappa\colon S\to M$ be a completely isometric, $G$-equivariant {u.c.p.} map. We then say that $(M,\kappa)$ is an \emph{extension} of $S$. Further, we say that
\begin{enumerate}[\upshape(i)]
\item $(M,\kappa)$ is \emph{$G$-essential} if it holds for any $G$-operator system $N$ and any $G$-equivariant \mbox{u.c.p.} map $\varphi\colon M\to N$ that $\varphi$ is completely isometric whenever $\varphi\circ\kappa$ is;
\item $(M,\kappa)$ is \emph{$G$-rigid} if the only $G$-equivariant \mbox{u.c.p.} map $\varphi\colon M\to M$ satisfying $\varphi\circ\kappa=\kappa$ is the identity map.
\end{enumerate}
If $(M,\kappa)$ is $G$-essential and $M$ is $G$-injective we say that $(M,\kappa)$ is a \emph{$G$-injective envelope} of $S$.
\end{defi}

If we let $G=\{1\}$ in the above definitions, we obtain the notions of \emph{injectivity} for operator systems, as well as plain \emph{essentiality} and \emph{rigidity} of an extension.

For any extension $(M,\kappa)$ of a $G$-operator system $S$ we will often suppress the complete isometry $\kappa$ and simply assume that $S$ is a $G$-invariant operator subsystem of $M$. Further, an extension $(M,\kappa)$ of a unital $C^*$-algebra $S$ is said to be a \emph{$C^*$-algebra extension} if $M$ is a $G$-$C^*$-algebra and $\kappa\colon S\to M$ is a $G$-equivariant, unital, injective $^*$-homomorphism.

\subsection{Equivariant injective envelopes}
In \cite[Section 2]{hamana1985} (see also \cite{hamana1979} and \cite{hamanaop}), Hamana proves that any $G$-operator system $S$ has a $G$-injective envelope $(I_G(S),\kappa)$ which is unique in the sense that for any other $G$-injective envelope $(M,\kappa')$ there is a complete order isomorphism $\varphi\colon I_G(S)\to M$ satisfying $\varphi\circ\kappa=\kappa'$. The \emph{injective envelope} $I(S)$ of an operator system $S$ can be taken to be the $G$-injective envelope of $S$ for $G=\{1\}$. 

We briefly sketch Hamana's proof. For any Hilbert space $H$ and any operator system $S\s B(H)$, the space $\ell^\infty(G,S)$ of bounded $S$-valued functions with coordinatewise operations becomes a $G$-operator subsystem of $B(H\otimes\ell^2(G))$ with the action of $G$ given by left translation, i.e.,
$$(gf)(h)=f(g^{-1}h),\quad g,h\in G,\ f\in\ell^\infty(G,S),$$ and each $f\in\ell^\infty(G,S)$ acting on $H\otimes\ell^2(G)$ by defining $f(\xi\otimes\delta_g)=f(g)\xi\otimes\delta_g$ for $\xi\in H$ and $g\in G$. Hamana first shows that if $S$ is an injective operator system, then $\ell^\infty(G,S)$ is $G$-injective, and that any $G$-injective $G$-operator system is injective. This proves to be essential in his construction of the $G$-injective envelope which can be described as follows. For $S\s M$, where $M$ is a $G$-injective $G$-operator system, an \emph{$S$-projection} on $M$ is an idempotent $G$-equivariant \mbox{u.c.p.} map $\varphi\colon M\to M$ satisfying $\varphi|_S=\idd_S$. A partial ordering on the set of $S$-projections on $ M$ can be defined by writing $\varphi\prec\psi$ for $S$-projections $\varphi,\psi\colon M\to M$ if $\varphi\circ\psi=\psi\circ\varphi=\varphi$.

Hamana uses a Zorn's lemma argument on the set of seminorms induced by $S$-projections on $M$ to show that there is a minimal $S$-projection $\varphi\colon M\to M$. Letting $\kappa\colon S\to M$ be the inclusion map, then $(\varphi(M),\kappa)$ is a $G$-rigid and $G$-injective extension of $S$. Since a $G$-injective extension is $G$-rigid if and only if it is $G$-essential, it follows that $(\varphi(M),\kappa)$ is the $G$-injective envelope of $S$. 

The canonical $G$-injective $G$-operator system is $\ell^\infty(G,B)$ where $B$ is an injective $C^*$-algebra. As the $G$-injective envelope is unique up to $G$-equivariant isomorphism, the preceding discussion of the construction of the $G$-injective envelope allows us to record the following lemma.

\begin{lemma}\label{injyes}Let $A$ be a unital $G$-$C^*$-algebra and let $B$ be an unital injective $C^*$-algebra containing $A$ as a unital $C^*$-subalgebra. Let $\kappa\colon A\to\mM=\ell^\infty(G,B)$ be the $G$-equivariant injective $^*$-homomorphism given by $$\kappa(x)(g)=g^{-1}x,\quad x\in A,\ g\in G.$$
Then there is a $\kappa(A)$-projection $\varphi\colon\mM\to\mM$ such that $(\varphi(\mM),\kappa)$ is the $G$-injective envelope of $A$.\end{lemma}
\begin{defi}For any injective extension $B$ of a unital $G$-$C^*$-algebra $A$, the map $\kappa\colon A\to\ell^\infty(G,B)$ defined in Lemma~\ref{injyes} will be referred to as the \emph{canonical inclusion map}.\end{defi}

Any injective operator system is unitally and completely order isomorphic to a unital, monotone complete $AW^*$-algebra (see \cite[Theorem~3.1]{choieffros} and \cite[Section 8.2]{saitowright}). For an introduction to $AW^*$-algebras, we refer to \cite{berberian}. 

In our setting, if $S\s M$ are as above and $\varphi\colon M\to M$ is a minimal $S$-projection, then the multiplication on $I_G(S)=\varphi(M)$ is given by the \emph{\mbox{Choi-Effros} product}, i.e., by $$x\circ y=\varphi(xy),\quad x,y\in I_G(S),$$ and the involution and norm on $I_G(S)$ are inherited from $M$ (see \cite[Section 2]{hamana1979} for results on idempotent \mbox{u.c.p.} endomorphisms of unital $C^*$-algebras). Further, if $A$ is a unital $G$-$C^*$-algebra, then $A$ embeds into the $G$-injective envelope as a $G$-invariant unital $C^*$-subalgebra.

In the case when $G=\{1\}$ the above product yields a $C^*$-algebra structure on the injective envelope $I(S)$ of $S$.

The following lemma is a $G$-equivariant version of \cite[Lemma~4.6]{hamana1979}.
\begin{lemma}\label{essencrit}Let $(M,\iota)$ be an extension of a $G$-operator system $S$ and let $\kappa\colon S\to I_G(S)$ denote the canonical inclusion. Then $(M,\iota)$ is $G$-essential if and only if there exists a unital $G$-equivariant complete isometry $\varphi\colon M\to I_G(S)$ such that $\varphi\circ\iota=\kappa$. Moreover, $I_G(S)$ is $G$-equivariantly completely order isomorphic to $I_G(M)$ whenever this holds.\end{lemma}
\begin{proof}If $(M,\iota)$ is a $G$-essential extension of $S$, then $G$-injectivity of $I_G(S)$ yields a $G$-equivariant \mbox{u.c.p.} map $\varphi\colon M\to I_G(S)$ satisfying $\varphi\circ\iota=\kappa$. Since $\kappa$ is a complete isometry, so is $\varphi$.

If there exists a completely isometric $G$-morphism $\varphi\colon M\to I_G(S)$ with $\varphi\circ\iota=\kappa$, assume that $\psi\colon M\to N$ is a $G$-equivariant \mbox{u.c.p.} map such that $\psi\circ\iota$ is completely isometric. Then there exists a $G$-equivariant \mbox{u.c.p.} map $\alpha\colon N\to I_G(S)$ such that $\alpha\circ\psi\circ\iota=\kappa$. By $G$-injectivity of $I_G(S)$ there is a $G$-equivariant \mbox{u.c.p.} map $\varphi'\colon I_G(S)\to I_G(S)$ such that $\varphi'\circ\varphi=\alpha\circ\psi$. Since
$$\varphi'\circ\kappa=\varphi'\circ\varphi\circ\iota=\alpha\circ\psi\circ\iota=\kappa,$$ $G$-rigidity yields $\varphi'=\idd_{I_G(S)}$, so that $\alpha\circ\psi=\varphi$. Since $\varphi$ is completely isometric, $\psi$ must be as well.

Whenever $\varphi\colon M\to I_G(S)$ is a $G$-equivariant complete isometry such that $\varphi\circ\iota=\kappa$, the inclusions $\kappa\colon S\to I_G(S)$ and $\kappa'\colon M\to I_G(M)$ allow us to construct $G$-equivariant \mbox{u.c.p.} maps $\tilde\varphi\colon I_G(M)\to I_G(S)$ and $\psi\colon I_G(S)\to I_G(M)$ such that $\tilde\varphi\circ\kappa'=\varphi$ and $\psi\circ\varphi=\kappa'$. Since $\tilde\varphi\circ\psi\circ\kappa=\kappa$ and $\psi\circ\tilde\varphi\circ\kappa'=\psi\circ\varphi=\kappa'$, $G$-rigidity yields that $\psi$ is a $G$-equivariant complete order isomorphism with inverse $\tilde\varphi$.\end{proof}

\begin{rem}\label{starhominj}If $M$ is a $G$-essential $C^*$-algebra extension of $S$, then the completely isometric $G$-morphism $\varphi\colon M\to I_G(S)$ of the above lemma can be chosen to be a $^*$-homomorphism, since $M$ is a $G$-invariant $C^*$-subalgebra of $I_G(M)=I_G(S)$.\end{rem}

\begin{lemma}\label{minimalsmall}Let $S$ be a $G$-operator system and let $M$ be a $G$-injective $G$-operator system with $S\s M$. If $\varphi\colon M\to M$ is an $S$-projection, there exists an $S$-projection $\psi\colon M\to M$ such that $\psi\prec\varphi$ and $\psi(M)=I_G(S)$.\end{lemma}
\begin{proof}Since $\varphi(M)$ is a $G$-injective $G$-operator system containing $S=\varphi(S)$, we may let $\Psi\colon\varphi(M)\to\varphi(M)$ be an $S$-projection such that $\Psi(\varphi(M))$ is the $G$-injective envelope of $S$. Define $\psi=\Psi\circ\varphi$.\end{proof}

\begin{rem}\label{nonGinj}Let $A$ be a unital $G$-$C^*$-algebra and let $\alpha\colon G\to\Aut(A)$ be the $G$-action. Writing $\alpha_g=\alpha(g)$ for all $g\in G$, then by injectivity each $\alpha_g\colon A\to A$ extends to a $^*$-isomorphism $I(A)\to I(A)$ which we will also denote by $\alpha_g$. Using rigidity one can show that $\alpha_g\circ\alpha_h=\alpha_{gh}$ on $I(A)$ for all $g,h\in G$, so that $I(A)$ becomes a unital $G$-$C^*$-algebra containing $A$ as a $G$-invariant $C^*$-subalgebra. Further, the inclusion $A\hookrightarrow I(A)$ is a $G$-essential extension of $A$ \cite[Remark 2.6]{hamana1985}.\end{rem}

In general the injective envelope of a unital $C^*$-algebra is not easily determined, even in the non-equivariant case. Some of the simplest cases when the injective envelope can be completely characterized, concern unital $C^*$-algebras that in some way generate an injective $C^*$-algebra monotonically.

\begin{prop}\label{injGX}Let $X$ be a discrete $G$-space. If $B$ is the unitization of $c_0(X)$, then $I(B)=\ell^\infty(X)$. 
If $H\s G$ is an amenable subgroup and $X=G/H$, then $I_G(B)=\ell^\infty(G/H)$, where $\ell^\infty(G/H)$ is equipped with the $G$-action by left translation.\end{prop}
\begin{proof}We adapt an argument of Paulsen \cite[Proposition 3.5]{paulsen2011} to prove the first claim. If $\phi\colon\ell^\infty(X)\to\ell^\infty(X)$ is a unital and positive map such that $\phi$ fixes $c_0(X)$, then any positive $f\in\ell^\infty(X)$ is the supremum of an increasing net of positive functions in $c_0(X)$. Thus $f\leq\phi(f)$. For $r>0$ such that $r1\geq f$, then $r1-f\leq\phi(r1-f)=r1-\phi(f)$, meaning that $f\geq\phi(f)$. Hence $\phi(f)=f$, so that $B\s\ell^\infty(X)$ is a rigid inclusion. Since $\ell^\infty(X)$ is injective, $I(B)=\ell^\infty(X)$.

If $H$ is a subgroup of $G$, we may identify $\ell^\infty(G/H)$ with a $G$-invariant $C^*$-subalgebra of $\ell^\infty(G)$. The second claim then follows from the fact that $\ell^\infty(G)$ is $G$-injective, and that there exists a $G$-equivariant conditional expectation $\ell^\infty(G)\to\ell^\infty(G/H)$ if and only if $H$ is amenable \cite[Corollary~4.4.5]{anantharamandelaroche}.\end{proof}

One of the key results of \cite{kalantarkennedy} is that the maximal ideal space of the $G$-injective envelope $I_G(\CC)$ is the \emph{Furstenberg boundary} $\partial_FG$. For more information on the Furstenberg boundary of a (not necessarily discrete) group, we refer to \cite{furstenberg}.

\subsection{The intersection property}

We recall that for any two $G$-$C^*$-algebras $A$ and $B$, any $G$-equivariant {u.c.p.} map $\pi\colon A\to B$ induces a $G$-equivariant {u.c.p.} map $\pi\rtimes_r\mr{id}\colon A\rtimes_r G\to B\rtimes_r G$ satisfying $(\pi\rtimes_r\mr{id})(a\lambda_s)=\pi(a)\lambda_s$ for all $a\in A$ and $s\in G$.

For any $G$-invariant ideal $I\s A$, let $I\rtimes_r G$ denote the ideal in $A\rtimes_r G$ generated by $I$. Further, if $\pi\colon A\to A/I$ is the quotient map, we define $I\bar\rtimes_r G$ to be the kernel of $\pi\rtimes_r\mr{id}\colon A\rtimes_r G\to (A/I)\rtimes_r G$. It is then evident that $I\rtimes_r G\s I\bar\rtimes_r G$.

The following notion was originally defined by Sierakowski \cite{sierakowski2010}.

\begin{defi}Let $G$ be a group and let $A$ be a $G$-$C^*$-algebra. We say that the $G$-action on $A$ is \emph{exact} if $I\rtimes_r G=I\bar{\rtimes}_r G$ for all $G$-invariant ideals $I\s A$.\end{defi}

A famous result by Kirchberg and Wassermann states that the reduced group $C^*$-algebra of a \emph{discrete} group $G$ is an exact $C^*$-algebra if and only if any action of $G$ on a $G$-$C^*$-algebra is exact (cf. \cite{kirchbergwassermann}).

The following notion, coined by Svensson and Tomiyama in \cite{svenssontomiyama}, has proven to be instrumental in unmasking the ideal structure of a reduced crossed product.

\begin{defi}Let $G$ be a group and let $A$ be a $G$-$C^*$-algebra. We say that the action of $G$ on $A$ has the \emph{intersection property} if every non-zero ideal of the reduced crossed product $A\rtimes_r G$ has non-zero intersection with $A$. The action of $G$ on $A$ is said to have the \emph{residual intersection property} if it holds for all $G$-invariant ideals $I\s A$ that the induced action of $G$ on $A/I$ has the intersection property.\end{defi}

Recall that the action of a group $G$ on a topological space $X$ is said to be \emph{free} if for all $x\in X$ and $g\in G$, $gx=x$ implies $g=1$, and \emph{topologically free} if no $g\in G\sm\{1\}$ fixes any non-empty open subset of $X$ pointwise. In \cite{kawamuratomiyama}, Kawamura and Tomiyama proved for any discrete, amenable group $G$ acting on a compact Hausdorff space $X$ that the action of $G$ on $C(X)$ has the intersection property if and only if the action of $G$ on $X$ is topologically free \cite[Theorem 4.1]{kawamuratomiyama}. In \cite{archboldsp} Archbold and Spielberg defined a generalized notion of topological freeness for actions of discrete groups on arbitrary $C^*$-algebras, in such a way that topological freeness still implied the intersection property.

The following structure theorem of Sierakowski \cite[Theorem~1.10]{sierakowski2010} gives a picture of exactly when the ideal structure of a re\-duced crossed pro\-duct can be completely understood. A $G$-$C^*$-algebra $A$ is said to \emph{separate the ideals} of $A\rtimes_r G$ if the map $I\mapsto I\rtimes_r G$ from the set of $G$-invariant ideals in $A$ to the set of ideals in $A\rtimes_r G$ is a bijection.
\begin{theorem}Let $A$ be a $G$-$C^*$-algebra. Then $A$ separates the ideals in $A\rtimes_r G$ if and only if the action of $G$ on $A$ is exact and has the residual intersection property.\end{theorem}
It was proved by M. Kennedy and the author in \cite{bryderkennedy} that for any $C^*$-simple discrete group $G$ (i.e., the reduced group $C^*$-algebra $C^*_r(G)$ is simple) and any $G$-$C^*$-algebra $A$, $A$ separates the \emph{maximal} ideals in $A\rtimes_r G$.

The class of exact groups is comparatively large -- indeed it was not known until 2003 whether non-exact groups did exist, when Gromov gave the first example of a finitely generated, non-exact discrete group, \cite{gromov}. Due to the above theorem and the abundance of exact groups (ensuring exactness of group actions), our aim will be to find criteria for the group action to satisfy the (residual) intersection property. As noted by Sierakowski in the aforementioned paper, the action of $G$ on $A$ has the residual intersection property whenever $I\s (I\cap A)\bar\rtimes_r G$ for any ideal $I\s A\rtimes_r G$ (equality holds when the action is also exact).

Recall for any $C^*$-algebras $A$ and $B$ and any completely positive map $\varphi\colon A\to B$ that the \emph{multiplicative domain} $\mult(\varphi)$ of $\varphi$ is given by \begin{align*}\mult(\varphi)&=\{a\in A\,|\,\varphi(a^*a)=\varphi(a)^*\varphi(a),\ \varphi(aa^*)=\varphi(a)\varphi(a)^*\}.\end{align*}As can be seen in \cite[Proposition~1.5.7]{brownozawa}, $\varphi$ in fact satisfies $$\varphi(ab)=\varphi(a)\varphi(b),\quad\varphi(ba)=\varphi(b)\varphi(a),\quad a\in\mult(\varphi),\ b\in A.$$

One key observation to be applied in the next section comes in the form of the following lemma, adapted from \cite[Lemma~7.3]{kalantarkennedyozbr}.

\begin{lemma}\label{google}Let $A$ be a $G$-$C^*$-algebra and let $X$ denote the maximal ideal space of $Z(A)$. If the action of $G$ on $X$ is free, then $I\s (I\cap A)\bar\rtimes_r G$ for all closed ideals $I\s A\rtimes_r G$.\end{lemma}

\begin{proof}Let $I_A=I\cap A$ and let $\pi\colon A\to A/I_A$ be the quotient map. Now let $\gamma\colon A/I_A\to B(H)$ be an irreducible representation of $A/I_A$ and consider the following representation
$$A+I\to (A+I)/I\cong A/I_A\stackrel{\gamma}{\to} B(H).$$ Due to Arveson's extension theorem this map extends to a {u.c.p.} map $\varphi\colon A\rtimes_r G\to B(H)$ such that $\varphi(I)=0$ and $A\s\mr{mult}(\varphi)$, since $\varphi|_A=\gamma\circ\pi$. By irreducibility, the restriction of $\varphi$ to $Z(A)$ is a point mass on $X$, i.e., $\varphi|_{Z(A)}=\delta_x$ for some $x\in X$. Letting $g\in G\sm\{1\}$, then there exists $f\in Z(A)\cong C(X)$ such that $f(g^{-1}x)\neq f(x)$, implying $$\varphi(\lambda_g)f(x)=\varphi(\lambda_gf)=\varphi(gf\lambda_g)=f(g^{-1}x)\varphi(\lambda_g).$$
Therefore $\varphi(\lambda_g)=0$. Letting $E\colon A\rtimes_r G\to A$ be the canonical conditional expectation (see, e.g., \cite[Proposition~4.1.9]{brownozawa}), it follows that $\varphi=\varphi\circ E$. Hence $$\gamma(\pi(E(I)))=\varphi(E(I))=\varphi(I)=\{0\}.$$ Since $\gamma$ was arbitrary, $\pi(E(I))=\{0\}$, so that $E(I)\s I_A$.

For any positive element $x\in I$, let $y$ be the image of $x$ under the $^*$-ho\-mo\-mor\-phism $\pi\rtimes_r\mr{id}\colon A\rtimes_r G\to (A/I_A)\rtimes_r G$ and let $E'\colon (A/I_A)\rtimes_r G\to A/I_A$ be the canonical faithful conditional expectation. Since $E'\circ(\pi\rtimes_r\mr{id})=\pi\circ E$, it follows that $E'(y)=0$ since $E(x)\in I_A$. As $E'$ is faithful, $y=0$ and $x\in I_A\bar\rtimes_r G$.\end{proof}

The following example is adapted from \cite[p.\ 9]{delaharpeskandalis}.
\begin{rem}\label{nonint}Let $G$ be a non-trivial discrete group, let $H$ be a non-tri\-vial amenable subgroup of $G$, and let $X=G/H$ be the left coset space. The unitization $B$ of the non-unital commutative $C^*$-algebra $c_0(X)$ becomes a unital $G$-$C^*$-algebra with exactly two proper $G$-invariant ideals, namely $\{0\}$ and $c_0(X)$. By Green's imprimitivity theorem, the ideal $J=c_0(X)\rtimes_r G$ of $B\rtimes_r G$ is Morita equivalent to $C^*_r(H)\cong C^*(H)$. As $C^*(H)$ has a proper non-zero ideal (e.g., the kernel of the trivial representation $C^*(H)\to\CC$), then so does $J$, say $I$, which is then also a proper non-zero ideal of $B\rtimes_r G$. However $I\cap B=\{0\}$, so the action of $G$ on $B$ does not have the intersection property.

As was also mentioned in \cite[Remark 4.10]{bryderkennedy}, $B$ is $G$-prime but $B\rtimes_r G$ may not be prime. For instance, if $H$ were abelian, then $C^*(H)$ contains two non-zero ideals that intersect only in $0$, so that the above-mentioned Morita equivalence yields two ideals $I$ and $J$ in $c_0(X)\rtimes_r G$ with $I\cap J=\{0\}$. We then note that $I$ and $J$ are also ideals in $B\rtimes_r G$.
\end{rem}

\subsection{Simplicity of reduced group \texorpdfstring{$C^*$}{C*}-algebras}
Our main reason for considering $G$-injective envelopes is the following result by Breuillard, Kalantar, Kennedy and Ozawa \cite[Theorem~3.1]{kalantarkennedyozbr} (see also \cite{kalantarkennedy}). Recall that a discrete group is said to be \emph{$C^*$-simple} if the reduced group $C^*$-algebra $C_r^*(G)$ is a simple $C^*$-algebra. 
\begin{theorem}\label{kk}A discrete group $G$ is $C^*$-simple if and only if the action of $G$ on the maximal ideal space $\partial_FG$ of $I_G(\CC)$ is free.\end{theorem}

\begin{rem}Let $G$ be a discrete group. Because $\ell^\infty(G/H)$ is $G$-injective for all amenable subgroups $H$ of $G$, as shown in Proposition \ref{injGX}, we observe that we can actually give a new proof of \cite[Theorem~5.3]{kennedy2015} of Kennedy. A subgroup $H$ of $G$ is said to be \emph{residually normal} if there exists a finite subset $F\s G\sm\{1\}$ such that
$F\cap gHg^{-1}\neq\oo$ for all $g\in G$. Kennedy proved that $G$ is $C^*$-simple if and only if it contains no amenable, residually normal subgroups. 

Suppose that $G$ has an amenable, residually normal subgroup $H$, and set $X=G/H$. By $G$-injectivity and $G$-essentiality, there exists a $G$-equivariant \mbox{u.c.p.} map $C(\partial_FG)\to\ell^\infty(X)$. If $\beta X$ is the Stone-\v{C}ech compactification of $X$, the dual map of this map restricts to a $G$-equivariant continuous map $\beta X\to\mathrm{Prob}(\partial_FG)$. By assumption, there exists a finite subset $F\s G\sm\{1\}$ such that $F\cap gHg^{-1}\neq\oo$ for all $g\in G$. Therefore, for all $g\in G$, there exists $s\in F$ and $h\in H$ such that $s=ghg^{-1}$. This implies $sgH=ghH=gH$, meaning that $gH\in X^s$, where $X^s$ denotes the set of elements in $X$ fixed by $s$. Therefore $X\s\bigcup_{s\in F}X^s$, so $\beta X\s\bigcup_{s\in F}\ov{X^s}=\bigcup_{s\in F}(\beta X)^s$, since $X$ is dense in $\beta X$. Letting $M\s\beta X$ be a minimal closed $G$-invariant subset, then $M\s\bigcup_{s\in F}M^s$. By \cite[Proposition 4.2]{furstenberg}, the map $\beta X\to\mathrm{Prob}(\partial_FG)$ restricts to a continuous $G$-equivariant surjection $M\to\partial_FG$, if we identify $\partial_FG$ with the subspace of point masses in $\mathrm{Prob}(\partial_FG)$. Thus $\partial_FG\s\bigcup_{s\in F}(\partial_FG)^s$, and since $F\s G\sm\{1\}$, the action of $G$ on $\partial_FG$ is not free, meaning that $G$ is not $C^*$-simple.

Conversely, if $G$ is not $C^*$-simple, then the action of $G$ on $\partial_FG$ is not topologically free \cite[Theorem~6.2]{kalantarkennedy}, so there exists $s\in G$ such that $(\partial_FG)^s$ has non-empty interior. By minimality of the action of $G$ of $\partial_FG$ \cite[Proposition~3.4]{kalantarkennedy} and compactness of $\partial_FG$, there is a finite set $F\s G\sm\{1\}$ such that $\partial_FG\s\bigcup_{t\in F}(\partial_FG)^t$. Fixing $x\in\partial_FG$, let $G_x=\{g\in G\,|\,gx=x\}$ denote the isotropy subgroup with respect to $x$. For any $g\in G$, there exists $s\in F$ such that $gx\in(\partial_FG)^s$, which implies $s\in G_{gx}=gG_xg^{-1}$. This proves that $G_x$ is residually normal, and since $G_x$ is amenable \cite[Proposition~2.7]{kalantarkennedyozbr}, we obtain Kennedy's result.\end{rem}

\section{The intersection property for essential extensions}

We now consider the question of whether the intersection property of an action of a discrete group $G$ on a $C^*$-algebra ascends to $G$-essential extensions. The proof of the following lemma is inspired by \cite[Lemma 7.2]{kalantarkennedyozbr}.

\begin{lemma}\label{ip1}Let $A$ be a $G$-$C^*$-algebra and let $B$ be a $G$-essential $C^*$-algebra extension of $A$. For any ideal $I\s A\rtimes_r G$, let $J$ be the ideal in $B\rtimes_r G$ generated by $I$. Then $I\cap A=\{0\}$ if and only if $J\cap B=\{0\}$.\end{lemma}

\begin{proof}The ``if'' part is trivial. 
For the converse, we may assume that $B$ is a unital $G$-invariant $C^*$-subalgebra of $I_G(A)$ containing $A$, due to Remark \ref{starhominj}. Let $\pi\colon A\rtimes_r G\to\mM$ be a $G$-equivariant unital $^*$-homomorphism with $\ker\pi=I$, where $\mM$ is a $G$-injective $G$-$C^*$-algebra (for instance, we may take $\mM=\ell^\infty(G,B(H))$ where $(A\rtimes_r G)/I$ is represented faithfully on $B(H)$). By $G$-injectivity, we can extend $\pi$ to a $G$-equivariant \mbox{u.c.p.} map $\tilde\pi\colon B\rtimes_r G\to\mM$.

Let $J$ be the ideal in $B\rtimes_r G$ generated by $I$ and assume that $I\cap A=\{0\}$. As $I_G(A)$ is $G$-injective and $\pi$ is completely isometric on $A$, there is a $G$-equivariant \mbox{u.c.p.} map $\varphi\colon\mM\to I_G(A)$ such that $\varphi\circ\tilde\pi|_A=\varphi\circ\pi|_A$ is the inclusion of $A$ into $I_G(A)$. Now let $\psi\colon I_G(A)\rtimes_r G\to\mM$ be a $G$-equivariant \mbox{u.c.p.} map such that $\psi|_{B\rtimes_r G}=\tilde\pi$, so that the following diagram commutes:
\[\begin{tikzcd}
A\arrow[d]\arrow[r,hook]\arrow[dd,hook,bend right=60]&B\arrow[d]\arrow[r,hook]&I_G(A)\arrow[d]\\
A\rtimes_r G\arrow[r]\arrow[dr,"\pi"]&B\rtimes_r G\arrow[r]\arrow[d,"\tilde\pi"]&I_G(A)\rtimes_r G\arrow[dl,dashed,"\psi"]\\
I_G(A)&\mM\arrow[l,dashed,"\varphi"]&\end{tikzcd}\]
By $G$-rigidity, then since $\varphi\circ\psi|_A$ is the inclusion map $A\to I_G(A)$, it follows that $\varphi\circ\psi|_{I_G(A)}=\idd_{I_G(A)}$. In particular, $\varphi\circ\tilde\pi|_B$ is the identity map on $B$. 
It follows that $\tilde\pi(B)\s\mr{mult}(\varphi)$, so that $\varphi$ is a $^*$-homomorphism on $C^*(\tilde\pi(B))$. Equipping $C^*(\tilde\pi(B))$ with the $G$-action given by conjugation by the unitaries $\pi(\lambda_g)$, $K=\ker(\varphi|_{C^*(\tilde\pi(B))})$ is a $G$-invariant ideal in $C^*(\tilde\pi(B))$. Now define $$D=C^*(\tilde\pi(B\rtimes_rG))=\ov{\Span}(C^*(\tilde\pi(B))\cdot\pi(C^*_r(G)))$$ and $$L=\ov{\Span}(K\cdot\pi(C^*_r(G))).$$
Both $D$ and $L$ are $G$-invariant $C^*$-subalgebras of $\mM$. For any $g,h\in G$ and $x,y\in C^*(\tilde\pi(B))$ we see that
$$x\pi(\lambda_g)y\pi(\lambda_h)=x(\pi(\lambda_g)y\pi(\lambda_g)^*)\pi(\lambda_{gh}).$$ Hence if either $x$ or $y$ belongs to $K$, then $x\pi(\lambda_g)y\pi(\lambda_h)\in L$, so $L$ is a $G$-invariant ideal of $D$.

Let $\phi\colon D\to D/L$ be the quotient map and let $(e_i)_{i\in I}$ be an approximate unit for $K$. Then $(e_i)_{i\in I}$ is an approximate unit for $L$ as well, and any $d\in D$ belongs to $L$ if and only if $e_id\to d$. Therefore $L\cap C^*(\tilde\pi(B))=K$. Now $\Phi=\phi\circ\tilde\pi\colon B\rtimes_r G\to D/L$ is multiplicative on $C^*_r(G)$, and since $\tilde\pi(x)^*\tilde\pi(x)-\tilde\pi(x^*x)\in\ker(\varphi|_{C^*(\tilde\pi(B))})\s L$ for all $x\in B$, it follows that $\Phi$ is a $^*$-homomorphism. Note further that $I\s\ker\Phi$, so that $J\s\ker\Phi$ as well, and that $\Phi$ is $G$-equivariant. Finally, if $\Phi(x)=0$ for $x\in B$ then $\tilde\pi(x)\in L\cap C^*(\tilde\pi(B))=K$. Thus $x=\varphi(\tilde\pi(x))=0$ and $\ker\Phi\cap B=\{0\}$. This completes the proof.
\end{proof}

The following theorem generalizes part of a result by Kawabe \cite[Theorem~3.4]{kawabe}.

\begin{theorem}\label{essip}Let $A$ be a $G$-$C^*$-algebra and let $B$ be a $G$-essential $C^*$-algebra extension of $A$. Then the action of $G$ on $A$ has the intersection property if and only if the action of $G$ on $B$ has the intersection property.\end{theorem}
\begin{proof}If $I\s A\rtimes_r G$ is an ideal such that $I\cap A=\{0\}$, then let $J\s B\rtimes_r G$ be the ideal generated by $I$. Since Lemma~\ref{ip1} yields that $J\cap B=\{0\}$, then if the action of $G$ on $B$ has the intersection property, it follows that $I\s J=\{0\}$. Conversely, if $J\s B\rtimes_r G$ is an ideal for which $J\cap B=\{0\}$, then $J\cap A=\{0\}$. Therefore if the action of $G$ on $A$ has the intersection property, we have $J\cap(A\rtimes_r G)=\{0\}$. By \cite[Theorem~3.4]{hamana1985} we may embed $B\rtimes_r G$ into $I(A\rtimes_r G)$ as a $G$-$C^*$-subalgebra. Since $xJ+Jx\s J$ for all $x\in A\rtimes_r G$, it follows from \cite[Lemma~1.2]{hamana1982} that $J=\{0\}$.\end{proof}

\section{Centres of injective envelopes}

To connect Theorem \ref{essip} to the discussion of the intersection property in Section 2, we now consider the centre of the $G$-injective envelope of a $G$-$C^*$-algebra.

Claim (i) of the following lemma has previously been noted in \cite[Proposition 6.4]{hamana1985}; we have not been able to locate a reference for claim (ii).

\begin{lemma}\label{goodlemma}Let $A$ be a $G$-$C^*$-algebra. Then:
\begin{enumerate}[\upshape(i)]
\item There is a $G$-equivariant unital injective $^*$-homomorphism $Z(A)\to Z(I_G(A))$.
\item If $A$ is $G$-injective, then so is $Z(A)$.
\end{enumerate}\end{lemma}
\begin{proof}Let $\mM=\ell^\infty(G,I(A))$, let $\kappa\colon A\to \mM$ be a $G$-equivariant inclusion and let $\varphi\colon\mM\to\mM$ be a $\kappa(A)$-projection so that $I_G(A)$ can be taken to be the image $\varphi(\mM)$ with the Choi-Effros product by Lemma~\ref{injyes}. For $x\in Z(\mM)$ and $y\in\varphi(\mM)$ \cite[Lemma~2.4]{hamana1979} yields $$\varphi(x)\circ y=\varphi(\varphi(x)y)=\varphi(xy)=\varphi(yx)=\varphi(y\varphi(x))=y\circ\varphi(x),$$ and thus $\varphi(Z(\mM))\s Z(I_G(A))$. To see that (i) holds, note that $Z(A)\s A'\cap I(A)=Z(I(A))$ by \cite[Corollary~4.3]{hamana1979}, so that $\kappa(Z(A))\s Z(\mM)$. Since $\kappa\colon A\to I_G(A)$ is a $^*$-homomorphism and $\varphi\circ\kappa=\kappa$, (i) follows.

For (ii), note first that $A$ is an $AW^*$-algebra, so $Z$ is a commutative $AW^*$-subalgebra of $A$ and is therefore injective \cite[Remark 2.5]{hadwinpaulsen}. Hence as in \cite[Remark 2.3]{hamana1985} we only need to construct a $G$-equivariant \mbox{u.c.p.} map $\varphi\colon\ell^\infty(G,Z)\to Z$ satisfying $\varphi\circ\kappa=\idd_Z$ where $\kappa$ is the canonical inclusion map -- it will then follow from $G$-injectivity of $\ell^\infty(G,Z)$ that $Z$ is $G$-injective.

Consider instead the canonical inclusion map $\kappa\colon A\to\mM=\ell^\infty(G,A)$ for $A$ and let $\varphi\colon\mM\to A$ be a $G$-equivariant \mbox{u.c.p.} map such that $\varphi\circ\kappa=\idd_A$, the existence of which follows from $G$-injectivity of $A$. 
Since $\kappa(A)\s\mult(\varphi)$, then for $z\in\ell^\infty(G,Z)=Z(\mM)$ we have $$\varphi(z)x=\varphi(z)\varphi(\kappa(x))=\varphi(z\kappa(x))=\varphi(\kappa(x)z)=\varphi(\kappa(x))\varphi(z)=x\varphi(z)$$ whenever $x\in A$, so $\varphi(z)\in Z$. Thus $\varphi$ maps $\ell^\infty(G,Z)$ into $Z$ and $\varphi\circ\kappa|_Z=\idd_Z$, so $Z$ is indeed $G$-injective.
\end{proof}

\begin{rem}\label{comm}Let $A$ be a unital commutative $G$-$C^*$-algebra. Then the $G$-injective envelope $I_G(A)$ is also commutative. We present a slightly different argument than the one given in \cite[Theorem~4.2]{hadwinpaulsen}. First note that the injective envelope $I(A)$ is commutative: indeed, we have $A\s A'\cap I(A)=Z(I(A))$ so for any unitary $u\in I(A)$ the automorphism $x\mapsto uxu^*$ of $I(A)$ is the identity map on $A$. Rigidity of $I(A)$ then implies $uxu^*=x$ for all $x\in I(A)$. 
If $\varphi\colon\ell^\infty(G,I(A))\to\ell^\infty(G,I(A))$ is a minimal $\kappa(A)$-projection, where $\kappa$ is the canonical inclusion map from Lemma~\ref{injyes}, then $I_G(A)$ is the image of $\varphi$ and the Choi-Effros product on $I_G(A)$ is evidently commutative.\end{rem}

\begin{exam}\label{nonfun}Any $G$-operator system $S$ admits a $G$-equivariant \mbox{u.c.p.} map $S\to I_G(\CC)$. Moreover, due to $G$-essentiality there always exists a completely isometric $G$-equivariant u.c.p. map $I_G(\CC)\to S$ whenever $S$ is $G$-injective. However, if $A$ is a unital $G$-injective $G$-$C^*$-algebra, there need not exist a $G$-equivariant unital injective $^*$-homomorphism $I_G(\CC)\to A$.

Indeed, let $B$ be the unital commutative $G$-$C^*$-algebra from Remark \ref{nonint} and let $A=I_G(B)$. Then $A$ is commutative, as verified in Remark \ref{comm}, and the action of $G$ on $A$ does not have the intersection property due to Theorem~\ref{essip}. By Lemma~\ref{google}, the action of $G$ on the maximal ideal space of $A$ is not free. If $G$ is non-trivial and $C^*$-simple, the action of $G$ on the maximal ideal space $\partial_FG$ of $I_G(\CC)$ is free by Theorem~\ref{kk}, so that there cannot exist a $G$-equivariant unital $^*$-homomorphism $I_G(\CC)\to A$.

In fact, we may apply Proposition \ref{injGX} to obtain that $A=\ell^\infty(G/H)$ for a non-trivial amenable subgroup $H$ of $G$. Any $G$-equivariant unital $^*$-homomorphism $I_G(\CC)\to\ell^\infty(G/H)$ gives rise to a $G$-equivariant continuous surjection $\beta(G/H)\to\partial_FG$. Since any element of $H$ fixes all points in $\beta(G/H)$, it follows that $H$ fixes all of $\partial_FG$ as well, contradicting that the action of $G$ on $\partial_FG$ is free.
\end{exam}

In light of Example \ref{nonfun}, we will now find a commutative $G$-injective envelope that the centre of a given $G$-injective envelope always does contain.

Recall that if $D$ is an $AW^*$-algebra, then any element $x\in D$ has a central support \cite[1.1.6]{berberian}, i.e., there exists a smallest central projection $C_x\in D$ such that $C_xx=x$. Further, if $z\in Z(D)$, then $zx=0$ if and only if $zC_x=0$.

\begin{lemma}\label{preparation-1}Let $A$ be an injective $G$-$C^*$-algebra with centre $Z$, and let $\kappa\colon A\to\ell^\infty(G,A)$ be the canonical inclusion map. If $\chi\colon\ell^\infty(G,Z)\to\ell^\infty(G,Z)$ is a $\kappa(Z)$-projection and $x\in A$, then $\kappa(x)\chi(f)=0$ whenever $\kappa(x)f=0$, for all $f\in\ell^\infty(G,Z)$.\end{lemma}
\begin{proof}Since $A$ is injective, it is an $AW^*$-algebra. By \cite[Proposition~1.1.10.1]{berberian} $\ell^\infty(G,A)$ is an $AW^*$-algebra as well. Let $\mZ=\ell^\infty(G,Z)$ and note that $\mZ$ is the centre of $\ell^\infty(G,A)$. Fix $x\in A$ and let $C_x\in Z$ be the central support of $x$. Then $\kappa(C_x)\in\mZ$ is the central support of $\kappa(x)$ in $\ell^\infty(G,A)$. Indeed, if $f\in\mZ$ is a projection such that $f\kappa(x)=\kappa(x)$, then $f(g)$ is a central projection in $A$ and $f(g)g^{-1}x=g^{-1}x$ for all $g\in G$. Thus $g^{-1}C_x\leq f(g)$ for all $g\in G$, so that $\kappa(C_x)\leq f$. Now, suppose that $\kappa(x)f=0$ for some $f\in\mZ$. Then $\kappa(C_x)f=0$. Since $\chi\circ\kappa|_Z=\kappa|_Z$, we have $\kappa(Z)\s\mult(\chi)$ and $$\kappa(C_x)\chi(f)=\chi(\kappa(C_x))\chi(f)=\chi(\kappa(C_x)f)=0.$$
Hence $\kappa(x)\chi(f)=0$.\end{proof}

\begin{lemma}\label{preparation-2}Let $A$ be an injective $G$-$C^*$-algebra with centre $Z$, and let $\mM=\ell^\infty(G,A)$ and $\mZ=\ell^\infty(G,Z)$.

Furthermore, let $\kappa\colon A\to\mM$ be the canonical inclusion map and $\Psi\colon A\otimes\mZ\to\mM$ be the $G$-equivariant $^*$-homomorphism given by $$\Psi(x\otimes f)=\kappa(x)f,\quad x\in A,\ f\in\mZ.$$
For any $\kappa(Z)$-projection $\chi\colon\mZ\to\mZ$ and any positive integer $n\geq 1$, let $\tilde{\Psi}=\Psi\circ(\idd_A\otimes\chi)$. Then the subset $$I_n=\{y\in A\otimes\mZ\otimes M_n(\CC)\,|\,(\tilde{\Psi}\otimes\idd_n)(y^*y)=(\tilde{\Psi}\otimes\idd_n)(yy^*)=0\}$$ is a closed ideal of $A\otimes\mZ\otimes M_n(\CC)$ containing $\ker(\Psi\otimes\idd_n)$.\end{lemma}
\begin{proof}To see that $I_n$ is an ideal, note for instance that for $y\in I_n$, $x\in A$, $f\in\mZ$ and $b\in M_n(\CC)$, then as $1\otimes f\otimes 1$ is central in $A\otimes\mZ\otimes M_n(\CC)$, we find that
\begin{align*}&(\varphi\otimes\idd_n)((x\otimes f\otimes b)yy^*(x\otimes f\otimes b)^*)\\&=(\varphi(a\otimes 1)\otimes b)(\tilde\Psi\otimes\idd_n)(y(1\otimes f^*f\otimes 1)y^*)(\varphi(a\otimes 1)\otimes b)^*\\&\leq\|f\|^2(\varphi(a\otimes 1)\otimes b)(\tilde\Psi\otimes\idd_n)(yy^*)(\kappa(x)\otimes b)^*=0,\end{align*} and $$(\tilde\Psi\otimes\idd_n)(y^*(x\otimes f\otimes b)^*(x\otimes f\otimes b)y)\leq\|x\|^2\|f\|^2\|b\|^2(\tilde\Psi\otimes\idd_n)(y^*y)=0,$$ which implies $(x\otimes f\otimes b)y\in I_n$.

If $x\in A$ and $z\in\mZ\otimes M_n(\CC)$ satisfy $(\Psi\otimes\idd_n)(x\otimes z)=0$, write $z^*z=\sum_{i,j}z_{ij}\otimes e_{ij}$, $z_{ij}\in\mZ$, with respect to the canonical basis $(e_{ij})$ of matrix units in $M_n(\CC)$. Then $\sum_{i,j}\kappa(x^*x)z_{ij}\otimes e_{ij}=0$, so that $\kappa(x^*x)z_{ij}=0$ for all $i,j$. By Lemma~\ref{preparation-1}, $\kappa(x^*x)\chi(z_{ij})=0$ for all $i,j$ as well, so that 
\begin{align*}&(\tilde{\Psi}\otimes\idd_n)((x\otimes z)^*(x\otimes z)))\\&=(\Psi\otimes\idd_n)((\idd_A\otimes\chi\otimes\idd_n)((x\otimes z)^*(x\otimes z)))\\
&=\sum_{i,j}\kappa(x^*x)\chi(z_{ij})\otimes e_{ij}=0.\end{align*}Therefore $x\otimes z\in I_n$ whenever $x\otimes z\in\ker(\Psi\otimes\idd_n)$ for $x\in A$ and $z\in\mZ\otimes M_n(\CC)$. Since $\mZ\otimes M_n(\CC)$ is exact, any closed ideal in $A\otimes(\mZ\otimes M_n(\CC))$ is generated by the elementary tensors it contains \cite[Propositions~2.16--2.17]{blanchardkirchberg}, so $\ker(\Psi\otimes\idd_n)\s I_n$.\end{proof}
\begin{theorem}\label{main2}Let $A$ be a unital $G$-$C^*$-algebra. Then there exists a $G$-essential $C^*$-algebra extension $B$ of $A$ such that $I_G(Z(I(A)))$ embeds into $Z(B)$ as a unital $G$-invariant $C^*$-subalgebra. In particular, $I_G(Z(I(A)))$ embeds into $Z(I_G(A))$ as a unital $G$-invariant $C^*$-subalgebra.\end{theorem}
\begin{proof}Since any $G$-essential extension of $I(A)$ is also a $G$-essential extension of $A$ by Lemma~\ref{essencrit} and Remark \ref{nonGinj}, we may assume that $A$ is injective. Let $$Z=Z(A),\ \mM=\ell^\infty(G,A),\  \mZ=\ell^\infty(G,Z)=Z(\mM),$$ and let $\kappa\colon A\to\mM$ be the canonical inclusion map. Then $\kappa(Z)\s\mZ$. Let $\Psi\colon A\otimes\mZ\to\mM$ be the $^*$-homomorphism of Lemma~\ref{preparation-2} and let $A_\mZ$ be the image of $\Psi$. Then $A_\mZ$ is a unital $G$-invariant $C^*$-subalgebra of $\mM$, and we obtain the following commutative diagram:
\[\begin{tikzcd}
Z\arrow[rr,"\kappa|_Z"] \arrow[dd,hook]&&\mZ\ar[dd,hook]\arrow[dl,"\hspace{-3pt}1\otimes\idd_\mZ"]\\
&A\otimes\mZ\arrow[dr,"\Psi"]&\\
A\arrow[ur,"\idd_A\otimes 1\hspace{-5pt}"]\arrow[rr,"\kappa"]&&A_\mZ\end{tikzcd}\]

Let $\Phi\colon\mM\to\mM$ be a $\kappa(A)$-projection so that $\Phi(\mM)=I_G(A)$. Since $\kappa(A)\s\mult(\Phi)$ it is easy to see that $\Phi(\mZ)\s\kappa(A)'=\mZ$, so by Lemma~\ref{minimalsmall} and Lemma~\ref{goodlemma} we can let $\chi\colon\mZ\to\mZ$ be a $\kappa(Z)$-projection such that $\chi(\mZ)=I_G(Z)$ and $\chi\prec\Phi|_\mZ$. 
Let $B\s\mM$ be the image of the $G$-equivariant \mbox{u.c.p.} map $$\tilde\Psi=\Psi\circ(\idd_A\otimes\chi)\colon A\otimes\mZ\to\mM.$$ Note that $(\idd_A\otimes 1)(A)\s\mult(\tilde\Psi)$, since $\tilde\Psi\circ(\idd_A\otimes 1)=\kappa$. Further, $\kappa$ maps $A$ into $B$ and 
$$\Phi(\tilde\Psi(x\otimes f))=\Phi(\kappa(x)\chi(f))=\kappa(x)\chi(f)=\tilde\Psi(x\otimes f),\quad x\in A,\ f\in\mZ.$$ This proves that $\Phi|_B=\idd_B$, so $B\s\Phi(\mM)=I_G(A)$ and therefore $(B,\kappa)$ is a $G$-essential extension of $A$ by Lemma~\ref{essencrit}.

We now claim that $B$ has the structure of a unital $G$-$C^*$-algebra. First endow the $G$-injective envelope $I_G(Z)$ with the Choi-Effros product of $\chi$, i.e., define $x\circ y=\chi(xy)$ for $x,y\in I_G(Z)$, so that $I_G(Z)$ is a unital $G$-$C^*$-algebra with the involution and norm of $\mZ$ and the product $\circ$.

If $y\in\mA\otimes\mZ$ and $\Psi(y)=0$, then $\tilde\Psi(y)^*\tilde\Psi(y)\leq\tilde\Psi(y^*y)=0$ by Lemma~\ref{preparation-2}. This allows us to define a unital $G$-equivariant map $\eps\colon A_\mZ\to A_\mZ$ with image $B$ by $$\eps(\Psi(y))=\tilde\Psi(y)=\Psi((\idd_A\otimes\chi)(y)),\quad y\in\mA\otimes\mZ,$$ so that the following diagram commutes:
\[\begin{tikzcd}
A\otimes\mZ\ar[r,"\Psi"]\ar[d,swap,"\idd_A\otimes\chi"]\ar[dr,"\tilde\Psi"]&A_\mZ\ar[r,hook]\ar[d,dashed,"\eps"]&\mM\ar[d,"\Phi"]\\A\otimes\mZ\ar[r,"\Psi"]&B\ar[r,hook]&I_G(A)\end{tikzcd}\]
Moreover, for $n\geq 1$ and any two $x,y\in A\otimes\mZ\otimes M_n(\CC)$ such that $(\Psi\otimes\idd_n)(x)=(\Psi\otimes\idd_n)(y)$ the lemma yields $(\tilde{\Psi}\otimes\idd_n)(x)=(\tilde\Psi\otimes\idd_n)(y)$. Since any positive element in $A_\mZ\otimes M_n(\CC)$ lifts via the $^*$-homomorphism $\Psi\otimes\idd_n$ to a positive element in $A\otimes\mZ\otimes M_n(\CC)$ for all $n$, $\eps$ is completely positive.
Finally, if we let $\tilde\chi=\idd_A\otimes\chi$, then $$\eps(\eps(\Psi(y)))=\eps(\Psi(\tilde\chi(y)))=\Psi(\tilde\chi(\tilde\chi(y)))=\Psi(\tilde\chi(y))=\eps(\Psi(y)),$$ so that $\eps$ is idempotent. As in the proof of \cite[Theorem~3.1]{choieffros} (see also \cite[Theorem~2.3]{hamana1979}), the image $B$ of $\eps$ is a unital $C^*$-algebra when endowed with the Choi-Effros product, i.e.,
$$x\ast y=\eps(xy),\quad x,y\in B.$$ This proves the claim. Furthermore, $B$ is completely order isomorphic to this unital $C^*$-algebra. Since $\kappa\colon A\to B$ is a $^*$-homomorphism with respect to the product on $B$, $B$ is a genuine $G$-essential $C^*$-algebra extension of $A$.

Let $\delta\colon I_G(Z)\to B$ be the map given by $\delta(x)=\Psi(1\otimes x)$. We claim that $\delta$ is in fact a unital $G$-equivariant injective $^*$-homomorphism of $(I_G(Z),\circ)$ into the centre of $(B,\ast)$. In fact, $\delta$ is the inclusion map $I_G(Z)=\chi(\mZ)\s\tilde\Psi(A_\mZ)=B$; we elect to use the above expression, as it makes calculations a bit neater. First, $\delta\circ\kappa|_Z=\kappa|_Z$ so that $\delta$ is a complete isometry, since $(I_G(Z),\kappa|_Z)$ is a $G$-essential extension of $Z$.
 Next, for $x,y\in\mZ$ we observe that
\begin{align*}\delta(\chi(x)\circ\chi(y))&=\Psi(1\otimes\chi(\chi(x)\chi(y)))\\
&=(\Psi\circ\tilde\chi)(1\otimes\chi(x)\chi(y))\\
&=\eps(\Psi(1\otimes\chi(x)\chi(y)))\\
&=\eps(\Psi(1\otimes\chi(x))\Psi(1\otimes\chi(y)))\\
&=\Psi(1\otimes\chi(x))\ast\Psi(1\otimes\chi(y))\\
&=\delta(\chi(x))\ast\delta(\chi(y)).\end{align*}
Therefore $\delta$ is a $^*$-homomorphism of $(I_G(Z),\circ)$ into $(B,\ast)$. Finally, for $x\in\mZ$ and $y\in A\otimes\mZ$ we have $1\otimes\chi(x)\in Z(A\otimes\mZ)$, and therefore \begin{align*}\delta(\chi(x))\ast\tilde\Psi(y)&=\eps(\Psi(1\otimes\chi(x))\tilde\Psi(y))\\&=\eps(\Psi((1\otimes\chi(x))\tilde\chi(y)))\\&=\eps(\Psi(\tilde\chi(y)(1\otimes\chi(x))))\\&=\eps(\Psi(\tilde\chi(y))\Psi(1\otimes\chi(x)))\\&=\tilde\Psi(y)\ast\delta(\chi(x)).\end{align*}
Thus $\delta$ is an injective $G$-equivariant $^*$-homomorphism into the centre of $B$. 

Finally, since $I_G(B)=I_G(A)$ by Lemma~\ref{essencrit} we see that $Z(B)$ embeds into $Z(I_G(B))=Z(I_G(A))$ as a unital $G$-invariant $C^*$-subalgebra due to Lemma~\ref{goodlemma}. This proves the second assertion.\end{proof}

\begin{rem}By Lemma \ref{goodlemma}, there always exists a $G$-equivariant injective $^*$-homomorphism $\iota\colon Z(I(A))\to Z(I_G(A))$ for any $G$-$C^*$-algebra $A$, since $I_G(I(A))=I_G(A)$. Scrutinizing the proofs of Lemma~\ref{goodlemma} and Theorem~\ref{main2}, we obtain a commutative diagram
\[\begin{tikzcd}
Z(I(A))\arrow[r,hook,"\iota"]\arrow[d,"\kappa"]&Z(I_G(A))\\
I_G(Z(I(A)))\arrow[ur,dashed,"\delta"]\end{tikzcd}\]
where $\delta\colon I_G(Z(I(A)))\to Z(I_G(A))$ is the $G$-equivariant injective $^*$-ho\-mo\-morphism obtained in the theorem, and $\kappa\colon Z(I(A))\to I_G(Z(I(A)))$ is the natural inclusion map.\end{rem}

\begin{rem}
If $A$ is a $G$-$C^*$-algebra and $Z=Z(I(A))$, the $G$-$C^*$-algebra $B$ constructed in the proof of Theorem~\ref{main2} is a $G$-equivariant quotient of $I(A)\otimes I_G(Z)$ such that the restriction of the quotient map $I(A)\otimes I_G(Z)\to B$ to either $I(A)$ or $I_G(Z)$ is an injective $^*$-homomorphism.

Indeed, let $\mZ=\ell^\infty(G,Z)$ and let $\chi\colon\mZ\to\mZ$ be the $\kappa(Z)$-projection from the proof of Theorem~\ref{main2}. Equipping $I_G(Z)$ with the Choi-Effros product arising from $\chi$, then the restriction $\eps_0$ to $I(A)\otimes I_G(Z)$ of the $G$-equivariant \mbox{u.c.p.} map $\eps\circ\Psi\colon I(A)\otimes\mZ\to B$ from the aforementioned proof is a surjective $^*$-homomorphism with respect to the new $C^*$-algebra structure of $I(A)\otimes I_G(Z)$. Indeed, \begin{align*}\eps_0(a\otimes\chi(f))\ast\eps_0(b\otimes\chi(g))&=\eps(\eps(\Psi(a\otimes\chi(f)))\eps(\Psi(b\otimes\chi(g))))\\&=\eps(\Psi(ab\otimes\chi(f)\chi(g)))\\&=\eps(\eps(\Psi(ab\otimes\chi(f)\chi(g))))\\&=\eps(\Psi(ab\otimes\chi(\chi(f)\chi(g)))\\&=\eps_0(ab\otimes(\chi(f)\circ\chi(g)))\\&=\eps_0((a\otimes\chi(f))(b\otimes\chi(g)))\end{align*} for all $a,b\in I(A)$ and $f,g\in\mZ$. Moreover, we may identify $I(A)\otimes I_G(Z)$ with the $G$-operator subsystem $D=(\idd_{I(A)}\otimes\chi)(I(A)\otimes\mZ)\s I(A)\otimes\mZ$ by means of a complete order isomorphism. Since $\Psi$ maps $D$ onto $B$ by definition, it follows that $\eps$ is surjective. Finally, $\eps_0(a\otimes 1)=\kappa(a)$ and $\eps_0(1\otimes\chi(f))=\chi(f)$ for all $a\in I(A)$ and $f\in\mZ$, proving the claim.

If $A$ is a prime $C^*$-algebra, then $I(A)$ has trivial centre \cite[Corollary~8.1.28]{saitowright}, and the $^*$-homomorphism $\Psi\colon I(A)\otimes\mZ\to\ell^\infty(G,I(A))$ is injective. It is easily verified that $\eps_0$ is then in fact a $^*$-isomorphism. Thus the $G$-$C^*$-algebra $I(A)\otimes I_G(\CC)$ is a $G$-essential extension of $A$ with respect to the inclusion map $a\mapsto a\otimes 1$, and by Lemma~\ref{essencrit} any $G$-invariant sub-$C^*$-algebra of $I(A)\otimes I_G(\CC)$ containing $A\otimes 1$ is a $G$-essential extension of $A$ as well, including $A\otimes I_G(\CC)$.\end{rem}

By Theorem~\ref{main2}, we obtain the following criterion for a group action on a $C^*$-algebra to have the intersection property.

\begin{theorem}\label{supercentre}Let $A$ be a unital $G$-$C^*$-algebra and let $I(A)$ denote the injective envelope of $A$. Whenever the action of $G$ on $Z(I(A))$ has the intersection property, then so does the action of $G$ on $A$ and $I(A)$.\end{theorem}
\begin{proof}Due to \cite[Theorem~3.4]{kawabe} the action of $G$ on the maximal ideal space of the $G$-injective envelope $I_G(Z(I(A)))$ is free, so by Theorem~\ref{main2} the action of $G$ on the maximal ideal space of $Z(I_G(A))$ is also free. Therefore the action of $G$ on $I_G(A)$ has the intersection property by Lemma~\ref{google}. The conclusion then follows from Theorem~\ref{essip}.\end{proof}

We next present some applications for the case when the group $G$ is $C^*$-simple.

\begin{cor}\label{minimalcentre}Let $G$ be a $C^*$-simple discrete group and let $A$ be a unital $G$-$C^*$-algebra. If $Z(I(A))$ is $G$-simple, then the action of $G$ on $A$ has the intersection property.\end{cor}
\begin{proof}By \cite[Theorem~7.1]{kalantarkennedyozbr}, the reduced crossed product $Z(I(A))\rtimes_r G$ is simple, so the action of $G$ on $Z(I(A))$ has the intersection property. Now apply Theorem~\ref{supercentre}.\end{proof}

\begin{cor}\label{primeloveit}Let $G$ be a $C^*$-simple discrete  group and let $A$ be a unital $G$-$C^*$-algebra. If $A$ is prime, then the action of $G$ on $A$ has the intersection property. In particular, $A\rtimes_r G$ is prime.\end{cor}
\begin{proof}Since $A$ is prime, $I(A)$ has trivial centre. By Corollary~\ref{minimalcentre}, the action of $G$ on $A$ has the intersection property.

If $J_1\cap J_2=\{0\}$ for ideals $J_1,J_2\s A\rtimes_r G$, then $(J_1\cap A)\cap (J_2\cap A)=\{0\}$ so $J_i\cap A=\{0\}$ and $J_i=\{0\}$ for some $i$. Therefore $A\rtimes_r G$ is prime.\end{proof}

In light of the above result it is worth noting that $C^*$-simplicity in itself need not transform $G$-primeness of a $C^*$-algebra to primeness of the reduced crossed product (see p.\ \pageref{nonint}).

\begin{cor}Let $G$ be a $C^*$-simple discrete group and let $A$ be a unital $G$-$C^*$-algebra. Then there is an injective map of the set of prime $G$-invariant ideals to the set of prime ideals in $A\rtimes_r G$, given by $I\mapsto I\bar\rtimes_r G$.\end{cor}
\begin{proof}If $I\s A$ is a prime, $G$-invariant ideal, then $A/I$ is a prime $C^*$-algebra and $(A/I)\rtimes_r G$ is a prime $C^*$-algebra by Corollary~\ref{primeloveit}. Thus $I\bar\rtimes_r G$ is a prime ideal of $A\rtimes_r G$, so the map $I\mapsto I\bar\rtimes_r G$ is well-defined, and it is injective since $(I\bar\rtimes_r G)\cap A=I$ for each $G$-invariant ideal $I\s A$.
\end{proof}

We will finally comment on the relationship between the centre of the equivariant injective envelope and the equivariant injective envelope of the centre of a $G$-$C^*$-algebra.

\begin{rem}\label{countcount}In \cite{hamana1981}, Hamana proves that every unital $C^*$-algebra $A$ has a \emph{regular monotone completion} $\bar{A}$, which is the smallest monotone complete $C^*$-subalgebra of $I(A)$ containing $A$. Furthermore, he proves that the centres $Z(\bar{A})$ and $Z(I(A))$ coincide. Since $\ov{Z(A)}$ is a monotone complete $C^*$-subalgebra of the commutative $C^*$-algebra $I(Z(A))$, it is injective, and thus $\ov{Z(A)}=I(Z(A))$.

In a later article \cite[Corollary 1.6]{hamana1982}, Hamana proves that the conditions (i) and (ii) below are equivalent for any unital $C^*$-algebra $A$. Notice that (i) and (i') are the same condition, due to the above discussion.
\begin{itemize}
\item[(i)] $\ov{Z(A)}=Z(\bar{A})$.
\item[(i')] $I(Z(A))=Z(I(A))$.
\item[(ii)] Every non-zero ideal $I$ of $A$ contains a non-zero positive element of $Z(A)$ whenever $I^{\perp\perp}=I$.
\end{itemize}
Here we define $S^\perp=\{x\in A\mid xS=Sx=\{0\}\}$ for any subset $S\s A$.

Sait\^{o} gave an early concrete example of a unital $C^*$-algebra $A$ for which $I(Z(A))\neq Z(I(A))$ in \cite{saito}. We adapt this example to the $G$-equivariant case as follows. Let $G$ be a discrete group, let $H$ be a separable infinite-dimensional Hilbert space, and let $A$ be the unitization of $\ell^\infty(G,K(H))\oplus\ell^\infty(G,K(H))$, where $K(H)$ denotes the non-unital $C^*$-algebra of compact operators on $H$. We may view $A$ as a unital $G$-invariant $C^*$-subalgebra of the von Neumann algebra $$\mM=\ell^\infty(G,B(H))\oplus\ell^\infty(G,B(H))$$ As $\mM$ is $G$-equivariantly $^*$-isomorphic to $\ell^\infty(G,B(H)\oplus B(H))$, it is $G$-injective \cite[Lemma 2.2]{hamana1985}. Moreover, just as any positive $x\in B(H)$ is the supremum of an increasing net of positive compact operators, any $f\in\ell^\infty(G,B(H))$ is the supremum of an increasing net of positive operators in $\ell^\infty(G,K(H))$. An argument similar to Paulsen's -- the one given in the proof of Proposition \ref{injGX} -- then shows that the inclusion $A\s\mM$ is $G$-rigid. Therefore $I_G(A)=\mM$, but $$I_G(Z(A))=I_G(\CC),$$ whereas $$Z(I_G(A))=Z(\mM)\cong\ell^\infty(G)\oplus\ell^\infty(G).$$
Since $I_G(\CC)$ is $G$-simple (as any $G$-equivariant unital $^*$-homomorphism from $I_G(\CC)$ is injective by $G$-essentiality) and $\ell^\infty(G)\oplus\ell^\infty(G)$ is not, we have $I_G(Z(A))\neq Z(I_G(A))$. However,  there \emph{does} exist a $G$-equivariant, unital $^*$-homomorphism $$\delta\colon I_G(Z(A))\to Z(I_G(A))\cong\ell^\infty(G,\CC\oplus\CC),$$ e.g., for characters $\varphi_1$ and $\varphi_2$ on $I_G(\CC)$ we may define $$\delta(x)(g)=(\varphi_1(g^{-1}x),\varphi_2(g^{-1}x)),\quad x\in I_G(\CC),\ g\in G.$$\end{rem}

\subsection*{Acknowledgements}
The author is grateful to \mbox{Magdalena} Musat and Mikael R\o rdam for their invaluable suggestions and comments, as well as their support.


\end{document}